\documentclass[11pt]{article}
\usepackage{latexsym, amsmath, amsfonts, amscd, amssymb, verbatim}
\usepackage{color}
\usepackage{hyperref}
\usepackage[margin=1in]{geometry}  % set the margins to 1in on all sides

\newtheorem{Theorem}{Theorem}[section]

\newtheorem{Remark}{Remark}[section]

\newtheorem{Lemma}{Lemma}[section]
\newtheorem{Corollary}{Corollary}[section]

\normalsize

\makeatletter
\renewcommand\@biblabel[1]{#1.}
\makeatother

\begin{document}

\title{\bf Lipschitz Continuity of Convex Functions}
\author{Bao Tran Nguyen\thanks{Universidad de O'Higgins, Rancagua, Chile and Quy Nhon University, Quy Nhon, Vietnam. E-mails: baotran.nguyen@uoh.cl; nguyenbaotran31@gmail.com} \and Pham  Duy Khanh\footnote{Department of Mathematics, HCMC University of Education, Ho Chi Minh, Vietnam and Center for Mathematical
		Modeling, Universidad de Chile, Santiago, Chile. E-mails:
		pdkhanh182@gmail.com; pdkhanh@dim.uchile.cl}}
\maketitle

\medskip
\begin{quote}
\noindent {\bf Abstract}  We provide some necessary and sufficient conditions for a proper lower semicontinuous convex function, defined on a real Banach space, to be locally or globally Lipschitz continuous. Our criteria rely on the existence of a bounded selection of the subdifferential mapping 
and the intersections of the subdifferential mapping and the normal cone operator to the domain of the given function. Moreover, we also point out that the Lipschitz continuity of the given function on an open and bounded (not necessarily convex) set can be characterized via the existence of a bounded selection  of the subdifferential mapping
 on the boundary of the given set and as a consequence it is equivalent to the local Lipschitz continuity at every point on the boundary of that set.  
 Our results are applied to extend a Lipschitz  and convex function to the whole space and to study the Lipschitz continuity of its Moreau envelope functions.
 

\end{quote}

\section{Introduction}
Lipschitz continuous and convex functions play a significant role in convex and nonsmooth analysis.
It is well-known that if the domain of a  proper lower semicontinuous convex function defined on a real Banach space has a nonempty interior then the function is continuous over the interior of its domain \cite[Proposition~2.111]{BS00} and as a consequence, it is subdifferentiable (its subdifferential is a nonempty set) and locally Lipschitz continuous at every point in the interior of its domain \cite[Proposition~2.107]{BS00}. Moreover, by the Hahn-Banach principle, at every interior point of the domain, the subdifferential of the given function is bounded by the corresponding Lipschitz constant.
Howerver, on the boundary of the domain, the subdifferential may be an empty set or unbounded one, and so the function may be not locally Lipschitz continuous. 

\medskip
For a convex function, the nonemptiness of its subdifferential at some point in its domain is equivalent to its calmness at that point \cite[Proposition~3.29]{Penot13}. Calmness is a property like local Lipschitz continuity but it involves comparisons only between a fixed point and nearby points, not between all possible pairs of points in some neighborhood of that fixed point. Therefore, for the calmness of a convex function, we can deal with boundary points of its domain. In light of this view, 
in this paper, we consider the notions of local and global Lipschitz continuity for a function with respect to its domain. Here, our points of Lipschitz continuity may not be in the interior of the domain of the given function. It turns out that local Lipschitz continuity is equivalent to the calmness on the corresponding neighborhood with the same modulus. This relationship helps us to characterize the local Lipschitz continuity of a proper convex function by the boundedness of a selection of its subdifferential. Moreover, by using the mean value theorem and the extension of a Lipschitz convex function to the whole space, we could establish some new characterizations for the local Lipschitz continuity for a lower semicontinuous convex function in terms of the intersections of its subdifferential and the normal cone to its domain. As a by product, some necessary and sufficient conditions for the global Lipschitz continuity are also deduced from the criteria for the local case. 

\medskip
Since the subdifferential of a proper lower semicontinuous convex function is nonempty on the interior of its domain, its Lipschitz continuity over an open set is equivalent to the boundedness of its subdifferential on that set. If the open set is bounded then its boundary is a nonempty set. It is interesting if we could get the Lipschitz continuity of the given function on an open bounded set from its information on the boundary that set.  
In this paper, by using the maximal monotonicity of the subdifferential operators, we can deduce the necessary and sufficient conditions for the Lipschitz continuity of a lower semicontinuous convex function
on an open bounded set from the boundedness of some selection of its subdifferential on the boundary of that set. 
This characterization allows us to show that the Lipschitz continuity on an open bounded set is equivalent to local Lipschitz continuity on the boundary of that set. We also use this characterization to estimate the values of the function on an open bounded set from the diameter of the given set and to deduce a criterion for the global Lipschitz continuity from the boundedness of the distance function from the origin to the subdifferential operator at infinity. 

\medskip
Our obtained results are used to study two classical problems in convex analysis: the extension of a Lipschitz and convex function to the whole space and the justification of Lipschitz continuity of its Moreau envelope functions.
For the first problem, Hiriart-Urruty \cite{Hiriat80} gave an explicit formula for the extension  by performing the infimal convolution of two functions associated with the data of the problem. We propose
here an alternative extension obtained by taking the supremum of all possible linear approximations of the given function. This extension is geometric and the subdifferential of the extending function is the biggest among all other extending ones. For the second problem, we show that the Lipschitz continuity of the 
 Moreau envelope functions can be inherited from the Lipschitz continuity of the original function. 
 It is interesting that all the regularized functions share the same Lipschitz constant of the original function. 
 
 \medskip
 The rest of the paper is structured as follows. Sect. 2 recalls some basic notations and preliminary results in convex analysis. Characterizations for the local and global Lipschitz continuity of a proper lower semicontinuous convex function are presented in Sect. 3. In Section. 4, necessary and sufficient conditions for the Lipschitz continuity on an open bounded set are investigated. The last section applies the obtained characterizations to the extension of a Lipschitz convex function and to establish the Lipschitz continuity of its Moreau envelope functions.

\section{Preliminaries} 
Let $X$ be a real Banach space with norm $\|\cdot\|$ and $X^*$ its continuous dual.  The value of a functional $x^*\in X^*$
at $x\in X$ is denoted by $\langle x^*,x\rangle$. The closed unit balls on $X$ and $X^*$ are denoted, respectively,  by 
$\mathbb{B}$ and $\mathbb{B}^*$.  For every $x\in X$ and $r>0$ the open ball with center $x$ and radius $r$ is given by
$$
B(x;r):=\{y\in X: \|x-y\|<r\}.
$$
The (\textit{effective}) \textit{domain} $\operatorname{dom} f$ of an extended real-valued function $f:X\rightarrow\overline{\mathbb{R}}:=\mathbb{R}\cup\{+\infty\}$ is the set of points $x$ where $f(x)\in\mathbb{R}$. The epigraph of $f$ is  defined by
$$
\operatorname{epi}f:=\{(x,r)\in X\times\mathbb{R}: f(x)\leq r\}.
$$
Recall that $f$ is \textit{proper} if its effective domain is nonempty and \textit{convex} if $\operatorname{epi}f$ is convex in $X\times\mathbb{R}$. Clearly, if $f$ is a proper convex function then $\operatorname{dom}f$ is a nonempty convex set. 
It is said that $f$ is \textit{lower semicontinuous} (\textit{l.s.c.}) if $\operatorname{epi}f$ is closed in $X\times\mathbb{R}$.   
Given a nonempty set $S$ of $X$ and a real number $\ell\geq 0$, $f$ is said to be \textit{Lipschitz continuous}  on $S$ with modulus $\ell$ or $\ell-$\textit{Lipschitz} on $S$
if $f$ is finite on $S$ and if
$$
	|f(x)-f(y)|\leq\ell\|x-y\|\quad \text{for all}\;\; x,y\;\;\text{in}\;\; S.
$$
The function $f$ is said to be \textit{locally Lipschitz} at $\bar{x}\in\operatorname{dom}f$ with the modulus $\ell\geq 0$ if there exists $r>0$ such that $f$ is $\ell-$Lipschitz on $\operatorname{dom}f\cap B(\bar{x};r)$. Local Lipschitz continuity is related to the concept of calmness. Recall that $f$ is \textit{calm  at} $\bar{x}\in \operatorname{dom}f$ with the modulus $\ell\geq 0$ 
(see \cite[p. 200]{Penot13} or \cite[p. 322]{RockafellarWets98})  if there is a neighborhood $V$ of $\bar{x}$ such that 
$$
f(x)\geq f(\bar{x})-\ell\|x-\bar{x}\|\quad \text{for all}\;\; x\in V.
$$
Clearly, if $f$ is locally Lipschitz continuous at $\bar{x}$ then $f$ is calm at $\bar{x}$ while the reverse is not true. 

\medskip
Lipschitz continuity of a function can be extended on the whole space with the same modulus (see  \cite[Theorem~1]{Hiriat80}).
\begin{Theorem}\label{ExstensionClassical}
	Let $f:X\rightarrow\overline{\mathbb{R}}$ be a $\ell-$\textit{Lipschitz} on a nonempty set $S\subset X$. 
	Consider the function $f_{S,\ell}:X\rightarrow\mathbb{R}$ given by
	$$
	f_{S,\ell}(x):=\inf_{u\in S}\{f(u)+\ell\|x-u\|\}\quad \text{for all}\quad x\in X.
	$$
	Then, $f_{S,\ell}$ is $\ell-$\textit{Lipschitz} on $X$ and coincides with $f$ on $S$. Moreover, if $S$ is a  convex set  and $f$ is convex then $f_{S,\ell}$ is convex. 
\end{Theorem}
Suppose now that $f$ is a proper convex function.
A functional $x^*\in X^*$ is said to be a \textit{subgradient} of $f$ at $x\in X$, if $f(x)$ is finite and 
$$
f(y)-f(x)\geq \langle x^*,y-x\rangle, \quad \forall y\in X.
$$
The collection of all subgradients of $f$ at $x$ is called the \textit{subdifferential} of $f$ at $x$, that is,
$$
\partial f(x):=\{x^*\in X^*: f(y)-f(x)\geq \langle x^*,y-x\rangle,\;\forall y\in X\}.
$$
The function $f$ is said to be \textit{subdifferentiable} at $x$ if $f(x)$ is finite and $\partial f(x)\ne\emptyset$.
The set of subdifferentiable points of $f$, denoted by $\operatorname{dom}\partial f$, is called the domain of $\partial f$. 
By \cite[Proposition~2.126~(iv)]{BS00}, if $f$ is continuous at $x\in\operatorname{dom}f$ then $x\in \operatorname{dom}\partial f$. Clearly, the mapping $\partial f$ is monotone, i.e., 
for every $x,y\in\operatorname{dom}\partial f$ we have
\begin{equation}\label{monotone}
\langle x^*-y^*,x-y\rangle\geq 0
\end{equation}
for all $x^*\in\partial f(x)$ and $y^*\in\partial f(y)$.
Moreover, by \cite[Theorem~A]{Rockafellar70}, if $f$ is l.s.c. then  $\partial f$ is maximal monotone, i.e., if $(x,x^*)\in X\times X^*$ satisfies the monotone relationship \eqref{monotone} for all $(y,y^*)\in X\times X^*$ with $y^*\in\partial f(y)$ then $x^*\in\partial f(x)$. Basic theory of maximal monotone operators
in Banach spaces is presented in \cite[Chapter~2]{Ba}.

The following theorem (see \cite[Theorem~3]{Rockafellar66}) gives a subdifferential sum rule for proper convex functions. 
\begin{Theorem} \label{SumRule}
	Let $f_1$ and $f_2$ be proper convex functions on $X$. Suppose that there exists a point at which both functions are finite and at least one is continuous. Then, for all $x\in X$,
$$
		\partial(f_1+f_2)(x)=\partial f_1(x)+\partial f_2(x).
$$
\end{Theorem}
We recall the Zagrodny mean value theorem \cite[Theorem~4.3]{Zagrodny88} for proper l.s.c. convex functions.
\begin{Theorem} \label{MeanValue} Let $f: X \to \overline{\mathbb{R}}$ be a proper l.s.c. convex function. For every $a,b\in \operatorname{dom}f$ with $a\ne b$, there are sequences $x_n\rightarrow_fc\in [a,b[$ and $x_n^*\in\partial f(x_n)$ satisfying 
\begin{itemize}
	\item[{\rm (a)}] $\displaystyle f(b)-f(a)\leq\liminf_{n\rightarrow\infty}\langle x_n^*,b-a\rangle$,
	\item[{\rm (b)}] $\displaystyle\frac{\|b-c\|}{\|b-a\|}(f(b)-f(a))\leq\liminf_{n\rightarrow\infty}\langle x_n^*,b-x_n\rangle$.
\end{itemize}
\end{Theorem}

Given a nonempty  set $\Omega\subset X$, int$\Omega$ is the \textit{interior} of $\Omega$, $\overline{\Omega}$ is the \textit{closure} of $\Omega$ and bd($\Omega$) is the \textit{boundary} of $\Omega$ with respect to strong topology on $X$. 
Suppose now that $\Omega$ is nonempty and convex. For every $x\in \Omega$ and $\varepsilon\geq 0$, we define the set of $\varepsilon-$\textit{normals} to $\Omega$ at $x$ (see, e.g., \cite[Definition~1.1 and Proposition~1.3]{Mordukhovich06})
by
\begin{equation}\label{EpsilonNormal}
N_\varepsilon(x;\Omega):=\{x^*\in X^*:\langle x^*, y-x\rangle\leq\varepsilon\|y-x\|\;\text{whenever}\; y\in \Omega\},
\end{equation}
When $\varepsilon=0$, elements of \eqref{EpsilonNormal} are called \textit{normals} and their collection, denoted by $N(x;\Omega)$, is the \textit{normal cone} to $\Omega$ at $x$. Clearly, for every $r>0$, $\Omega\cap B(x;r)$ is convex and 
\begin{equation}\label{localnormal}
N(x;\Omega)=N(x;\Omega\cap B(x;r)).
\end{equation}
The function $\delta_\Omega: X\rightarrow\overline{\mathbb{R}}$ defined by 
$$
\delta_\Omega(x):=
\begin{cases}
0 & \text{if}\; x\in \Omega,\\
+\infty & \text{otherwise},
\end{cases}
$$
is called the \textit{indicator function} of $\Omega$. Clearly, $\delta_\Omega$ is a proper convex function and for every $x\in \Omega$ we have $\partial\delta_{\Omega}(x)=N(x;\Omega)$. Finally, by using the subdifferential sum rule for  the fucntion $f_1(x)=\delta_\Omega(x)$ and $f_2(x)=\varepsilon\|x-\bar{x}\|$ at $\bar{x}$, we get
the following relationship
\begin{equation}\label{sumnormal}
N_\varepsilon(x;\Omega)=N(x;\Omega)+\varepsilon\mathbb{B}^*.
\end{equation}
\section{Local and global Lipschitz continuity}
We start with a necessary condition and a sufficient one  for the Lipschitz continuity of a proper convex function on a nonempty subset of its domain. The following lemma is simple but it is useful in the sequel. 
\begin{Lemma}\label{CharacterizationSet} Let $f: X \to \overline{\mathbb{R}}$ be a proper convex function, $\ell\geq 0$ and $S\subset\operatorname{dom}f$ a nonempty set. If
\begin{equation}\label{Intersection}
	\partial f(x)\cap \ell\mathbb{B}^*\ne\emptyset \;\;\text{for all}\;\; x\in S
\end{equation}
then $f$ is $\ell-$Lipschitz on $S$. Conversely, if $S$ is open and $f$ is $\ell-$Lipschitz on $S$ then
\begin{equation}\label{Inclusion}
\emptyset\ne\partial f(x)\subset \ell\mathbb{B}^* \;\;\text{for all}\;\; x\in S.
\end{equation}
\end{Lemma}
\begin{proof} Assume that \eqref{Intersection} is satisfied. We will prove  $f$ is $\ell-$Lipschitz on $S$. 
	It follows from \eqref{Intersection} that $S\subset\operatorname{dom}\partial f\subset\operatorname{dom}f$. Let $x,y$ be two points in $S$. Again, by \eqref{Intersection}, there exists $x^*\in\partial f(x)$ such that $\|x^*\|\leq\ell$. Employing the definition of subdifferential, we get
	$$
	f(y)-f(x)\geq\langle x^*,y-x\rangle\geq-\|x^*\|\|y-x\|\geq-\ell\|y-x\|.
	$$
	Changing the role of $x$ and $y$, we immediately obtain
	$$
	|f(y)-f(x)|\leq \ell\|y-x\|.
	$$
Suppose now that $S$ is open and $f$ is $\ell-$Lipschitz on $S$. 
Clearly, $f$ is continuous on $S$ and so \eqref{Inclusion}  is satisfied by  \cite[Proposition~2.126~(iv)]{BS00}.
	$\hfill\Box$
\end{proof}

\begin{Remark} {\rm It follows from Lemma~\ref{CharacterizationSet} that if $S$ is open then $f$ is $\ell-$Lipschitz on $S$ if and only if either \eqref{Intersection} or \eqref{Inclusion}  is satisfied. 
		However, $\ell-$Lipschitz continuity of $f$ on an arbitrary set $S$ can not imply  \eqref{Intersection}. For example, for $S=\{1\}\subset\mathbb{R}$ and $f(x)=x$, $f$ is a proper convex and $0-$Lipschitz function 
		on $S$ while $\partial f(1)\cap 0\mathbb{B^*}=\emptyset$. 
		}
\end{Remark}
\medskip
Characterization for the calmness at one point and Lemma~\ref{CharacterizationSet} allow us to find a necessary and sufficient condition for the local Lipschitz continuity of a proper convex function and to establish the  equivalence of its local Lipschitz continuity and its calmness at every point in the corresponding neighborhood. 
\begin{Theorem}\label{NonLSC} Let $f: X \to \overline{\mathbb{R}}$ be a proper convex function,
	 $\bar{x}\in\operatorname{dom}f$ and $r>0, \ell\geq 0$. Then, $f$ is $\ell-$Lipschitz on $\operatorname{dom}f\cap B(\bar{x};r)$  if and only if one of the following conditions holds
\begin{itemize}
	\item[{\rm (a)}] $f$ is calm at every point in $\operatorname{dom}f\cap B(\bar{x};r)$ with the same modulus $\ell$.
	\item[{\rm (b)}] $\partial f(x)\cap\ell\mathbb{B}^*\ne\emptyset
	$ for all $x\in \operatorname{dom}f\cap B(\bar{x};r)$.
\end{itemize}
\end{Theorem}
\begin{proof} Clearly, if $f$ is $\ell-$Lipschitz on $\operatorname{dom}f\cap B(\bar{x};r)$ then $f$  is calm at every point in the latter set with the same modulus $\ell$. Moreover, (a) implies (b) by \cite[Exercise 3, page 212]{Penot13} (see also \cite[Proposition~5.1]{MMS19}).  By Lemma~\ref{CharacterizationSet}, (b) implies that $f$ is $\ell-$Lipschitz on $\operatorname{dom}f\cap B(\bar{x};r)$.
$\hfill\Box$
\end{proof}

\medskip
As a direct application of Theorem~\ref{NonLSC}, we get the criteria for the global Lipschitz continuity of a proper convex function.
\begin{Corollary}\label{Global} Let $f: X \to \overline{\mathbb{R}}$ be a proper convex function and $\ell\geq 0$. Then, $f$ is $\ell-$Lipschitz on $\operatorname{dom}f$   if and only if one of the following conditions holds
\end{Corollary}
\begin{itemize}
	\item[{\rm (a)}] $f$ is locally Lipschitz at every point in $\operatorname{dom}f$ with the same modulus $\ell$.
	\item[{\rm (b)}] $f$ is calm at every point in $\operatorname{dom}f$ with the same modulus $\ell$.
	\item[{\rm (c)}] $\partial f(x)\cap\ell\mathbb{B}^*\ne\emptyset
	$ for all $x\in \operatorname{dom}f$.
\end{itemize}
\begin{Remark} {\rm Since the subdifferential of a proper convex function may be unbounded at boundary points of its domain, the $\ell-$Lipschitz continuity of $f$ on $\operatorname{dom}f$  can not imply
	$$
	\partial f(x)\subset \ell\mathbb{B}^*, \quad \forall x\in\operatorname{dom}f.
	$$
Indeed, the function $f(x)=\delta_{[0,1]}(x)$ is proper convex and $0-$Lipschitz function on $\operatorname{dom}f=[0,1]$ and its subdifferential is given by
$$
\partial f(x)=
\begin{cases}
\{0\} & \text{if } x\in (0,1),\\
(-\infty,0] &\text{if } x=0,\\
[0,+\infty) &\text{if } x=1.
\end{cases}
$$
Clearly, $\partial f(x)\subset 0\mathbb{B}^*$ when $x\in (0,1)$ while $\partial f(x)\not\subset 0\mathbb{B}^*$
when $x\in\{0,1\}$.
}
\end{Remark}

\medskip
We come to characterizations of the local Lipschitz continuity of a proper l.s.c. convex function in terms of 
the intersections of its subdifferential and the normal cone to its domain. 
\begin{Theorem}\label{LipschitzCharacterization}  Let $f: X \to \overline{\mathbb{R}}$ be a proper l.s.c. convex function,  $\bar{x}\in\operatorname{dom}f$, and $r>0, \ell\geq 0$.  Then, $f$ is $\ell-$Lipschitz on  $\operatorname{dom}f\cap B(\bar{x};r)$ if and only if one of the following conditions holds
		\begin{itemize}
			\item[{\rm (a)}]  
			$\emptyset\ne\partial f(x) \subset N_{\ell}(x; \operatorname{dom}f)$ for all
			$x\in\operatorname{dom}f\cap B(\bar{x};r)$.
			\item[{\rm (b)}]  
			$\partial f(x) \cap N_{\ell}(x; \operatorname{dom}f)\ne\emptyset$ for all
			 $x\in\operatorname{dom}f\cap B(\bar{x};r)$ .
		\end{itemize}
\end{Theorem}
\begin{proof} Since $f$ is proper and convex, $\operatorname{dom}f$ is a nonempty convex set. Suppose that $f$ is $\ell-$Lipschitz on $\operatorname{dom}f\cap B(\bar{x};r)$.
	By Theorem~\ref{NonLSC}, $f$ is subdifferentiable on $\operatorname{dom}f\cap B(\bar{x};r)$.
Invoking Theorem~\ref{ExstensionClassical}, we can construct an $\ell-$Lipschitz convex function $\tilde{f}:X\rightarrow\mathbb{R}$ such that
$$
\tilde{f}(x)=f(x), \quad \forall x\in \operatorname{dom}f\cap B(\bar{x};r).
$$
Since $\tilde{f}$ is $\ell-$Lipschitz continuous on $X$, we have
$\partial\tilde{f}(x)\subset \ell\mathbb{B}^*$ for all $x\in X$.
Hence, for every $x\in \operatorname{dom}f\cap B(\bar{x};r)$, applying Theorem~\ref{SumRule} 
for the convex functions $\tilde{f}$ and $\delta_{\operatorname{dom}f\cap B(\bar{x};r)}$ 
and using \eqref{localnormal} and \eqref{sumnormal}, we get the following 
inclusions
\begin{eqnarray*}
\partial f(x)&\subset&\partial\left(\tilde{f}+\delta_{\operatorname{dom}f\cap B(\bar{x};r)}\right)(x)\\
&=&\partial\tilde{f}(x)+N(x;\operatorname{dom}f\cap B(\bar{x};r))\\
&\subset&\ell\mathbb{B}^*+N(x;\operatorname{dom}f)\\
&=& N_{\ell}(x; \operatorname{dom}f).
\end{eqnarray*}
The implication ${\rm (a)}\Rightarrow{\rm (b)}$ is trivial. 

Suppose now that (b) holds, i.e., for every $x\in\operatorname{dom}f\cap B(\bar{x};r)$ we have
$\partial f(x) \cap N_{\ell}(x; \operatorname{dom}f)\ne\emptyset$.
Let $x,y\in  \operatorname{dom}f\cap B(\bar{x};r)$ be such that $x\ne y$ and let z be any point in $]x,y[$. Applying Theorem~\ref{MeanValue} for the l.s.c. proper convex $f$ and two distinct points $x,z$, we can find sequences $x_n\rightarrow v\in[x,z[$ and $x_n^*\in\partial f(x_n)$ such that
\begin{equation}\label{liminf}
	\frac{\|z-v\|}{\|z-x\|}(f(z)-f(x))\leq\liminf_{n\rightarrow\infty}\langle x_n^*,z-x_n\rangle.
\end{equation}
Let $z^*$  be any point in $\partial f(z) \cap N_{\ell}(z; \operatorname{dom}f)$. 
By the monotonicity of $\partial f$, for every $n\in\mathbb{N}$, we have
\begin{eqnarray*}
	\langle x_n^*,z-x_n\rangle&\leq&\langle z^*,z-x_n\rangle\\
	&=&\langle z^*,z-v\rangle+ \langle z^*,v-x_n\rangle\\
	&=&\frac{\|z-v\|}{\|y-z\|}\langle z^*,y-z\rangle+ \langle z^*,v-x_n\rangle\\
	&\leq&\frac{\|z-v\|}{\|y-z\|}\ell\|y-z\|+ \langle z^*,v-x_n\rangle\\
	&=&\ell\|z-v\|+ \langle z^*,v-x_n\rangle.
\end{eqnarray*}
It follows from \eqref{liminf} that 
$$
f(z)-f(x)\leq\ell\|z-x\|. 
$$
Taking $z\rightarrow y$ in the above inequality and using the lower semicontinuity of $f$, we get
$$
f(y)-f(x)\leq\ell\|y-x\|
$$
and so $f$ is $\ell-$Lipschitz on $\operatorname{dom}f$.
$\hfill\Box$
\end{proof}

\medskip
Employing Theorem~\ref{LipschitzCharacterization}, we derive the necessary and sufficient conditions for the global Lipschitz continuity of a proper l.s.c. convex function. 
\begin{Corollary}\label{normal} Let $f: X \to \overline{\mathbb{R}}$ be a proper l.s.c. convex function and $\ell\geq 0$. Then, $f$ is $\ell-$Lipschitz on $\operatorname{dom}f$  if and only if one of the following conditions holds
\begin{itemize}
	\item[{\rm (a)}]  
	$\emptyset\ne\partial f(x) \subset N_{\ell}(x; \operatorname{dom}f)$ for all
	$x\in\operatorname{dom}f$.
	\item[{\rm (b)}]  
	$\partial f(x) \cap N_{\ell}(x; \operatorname{dom}f)\ne\emptyset$ for all
	$x\in\operatorname{dom}f$ .
\end{itemize}
\end{Corollary}
\section{Lipschitz continuity on bounded and open subsets}
In this section we obtain characterizations of Lipschitz continuity of a proper l.s.c. convex function on an open and bounded (not necerrarily convex) set in term of the information on the boundary of the given set. 
We begin with a criterion based on the existence of a bounded selection of the subdifferential operator. 
\begin{Theorem} \label{LipschitzBoundary} Let $f: X \to \overline{\mathbb{R}}$ be a proper l.s.c. convex function, $\ell\geq 0$ and $S$ be a nonempty set such that $\overline{S}\subset\operatorname{dom}f$. Consider the following statements
	\begin{itemize}
		\item[{\rm (a)}] $\partial f(x)\cap \ell\mathbb{B}^*\ne\emptyset$ for all $x\in\operatorname{bd}(S)$. 
		\item[{\rm (b)}] $f$ is $\ell-$Lipschitz on $S$.
	\end{itemize}
If $S$ is bounded then {\rm (a)}$\Rightarrow${\rm (b)}. Conversely,  if  $S$ is open then  {\rm (b)}$\Rightarrow${\rm (a)}.
\end{Theorem}
\begin{proof} Suppose that $S$ is bounded and (a) holds. Let $x,y$ be two distinct points in $S$. We will show that
	\begin{equation}\label{left}
		f(x)-f(y)\geq -\ell\|x-y\|
	\end{equation}
	and so $f$ is $\ell-$Lipschitz on $S$. 
	Indeed, if $y\in\operatorname{bd}(S)$ then, by (a), there exists $y^*\in\partial f(y)$ such that $\|y^*\|\leq\ell$. By the definition of subdifferential we have
	$$
	f(x)-f(y)\geq \langle y^*,y-x\rangle\geq-\|y^*\|\|x-y\|\geq-\ell\|x-y\|
	$$  
and so \eqref{left} is satisfied.	If $y\in\text{int}S$ then $y$ is an interior point of $\operatorname{dom}f$ and thus $\partial f(y)\ne\emptyset$.  
By the boundedness of $S$,
there exists $\alpha>0$ such that $z:=y+\alpha(y-x)\in \operatorname{bd}(S)$. 
Again, by (a), we can find $z^*\in\partial f(z)$ such that $\|z^*\|\leq\ell$. 
Picking $y^*\in\partial f(y)$ and using the monotonicity of $\partial f$, we have 
\begin{equation}\label{Monotone}
\langle y^*,x-y\rangle=(1/\alpha)\langle y^*,y-z\rangle\geq(1/\alpha)\langle z^*,y-z\rangle=
\langle z^*,x-y\rangle
\end{equation}
Furthermore, since  $\|z^*\|\leq\ell$, we have
\begin{equation}\label{CauchySchwarz}
\langle z^*,x-y\rangle\geq-\|z^*\|\|x-y\|\geq-\ell\|x-y\|. 
\end{equation}
From the inequalities \eqref{Monotone} and \eqref{CauchySchwarz}, we get \eqref{left} by the following estimates
$$
f(x)-f(y)\geq\langle y^*,x-y\rangle\geq\langle z^*,x-y\rangle\geq -\ell\|x-y\|.
$$

Suppose now that $S$ is open and $f$ is $\ell-$Lipschitz on $S$. 
Let $x\in\operatorname{bd}(S)$  and  $\{x_n\}\subset S$ be a sequence such that $x_n\rightarrow x$. 
Since $S$ is open and $f$ is $\ell-$Lipschitz on $S$, by Lemma~\ref{CharacterizationSet}, there exists a sequence $\{x^*_n\}\subset X^*$ such that $x_n^*\in\partial f(x_n)$ and $x^*_n\in \ell\mathbb{B}^*$ for all $n\in\mathbb{N}$.
By the Banach-Alaoglu theorem, there exist $x^*\in X^*$ and subnet $\{x^*_i\}_{i\in I}$ of $\{x^*_n\}$ such that $\{x^*_i\}_{i\in I}$ is  weakly-star convergent to $x^*$ and $\|x^*\|\leq\ell$. By the maximal monotonicity of $\partial f$, we have $x^*\in\partial f(x)$ (see \cite[Fact 3.5]{BorweinYao13} or \cite[Section 2, page 539]{BFG03}) and so (a) holds. 
$\hfill\Box$  
\end{proof}

\begin{Remark} {\rm In Theorem~\ref{LipschitzBoundary}, if $S$ is  unbounded then (a)$\Rightarrow$(b) does not hold. Indeed, for $f(x)=x^2$ and $S=[0,+\infty)$ we have $\operatorname{bd}(S)=\{0\}$ and for every $\ell\geq 0$ we have
		$$
		\partial f(x)\cap\ell\mathbb{B}^*=\{0\}\ne\emptyset, \quad \forall x\in \operatorname{bd}(S).
		$$
	However, $f$ is not Lipschitz on $S$. 
		}
\end{Remark}

\medskip
With the help of Theorem~\ref{LipschitzBoundary}, we will show that the Lipschitz continuity on an open bounded set is equivalent to local Lipschitz continuity on the boundary of the corresponding set. 
\begin{Theorem} Let $f: X \to \overline{\mathbb{R}}$ be a proper l.s.c. convex function, $\ell\geq 0$ and $S$  be a nonempty open and bounded set such that $\overline{S}\subset\operatorname{dom}f$. Then, $f$ is $\ell-$Lipschitz on $S$ 
	if and only if $f$ is locally $\ell-$Lipschitz with respect to $S$ at every point in $\operatorname{bd}(S)$, i.e., for every $x\in\operatorname{bd}(S)$, there exists $r>0$ such that $f$ is $\ell-$Lipschitz on $S\cap B(x;r)$.
\end{Theorem}
\begin{proof} Since Lipschitz continuity on $S$ is equivalent to Lipschitz continuity on $\overline{S}$ with the same modulus, we only need to show local Lipschitz continuity on $\operatorname{bd}(S)$ with the same modulus implies Lipschitz continuity on $S$. Suppose now that $f$ is locally $\ell-$Lipschitz with respect to $S$ at every point in $\operatorname{bd}(S)$. We will use Theorem~\ref{LipschitzBoundary} to show that $f$ is $\ell-$Lipschitz on $S$. 
	Let $x$ be any point in $\operatorname{bd}(S)$. Then, there exists $r>0$ such that $f$ is $\ell-$Lipschitz on nonempty open set $S\cap B(x;r)\subset\operatorname{dom}f$. Therefore,  by Lemma~\ref{CharacterizationSet}, we have
	  $$
	  \partial f(y)\ne\emptyset, \quad \partial f(y)\subset\ell\mathbb{B}^*, \quad \forall y\in S\cap B(x;r).
	  $$
Let $\{x_n\}\subset S$ and $\{x^*_n\}\subset X^*$ be the sequences such that $x_n\rightarrow x$ and $x^*_n\in\partial f(x_n), \|x^*_n\|\leq\ell$ for all $n\in\mathbb{N}$. By the Banach-Alaoglu theorem, there exist $x^*\in X^*$ and subnet $\{x^*_i\}_{i\in I}$ of $\{x^*_n\}$ such that $\{x^*_i\}_{i\in I}$ is  weakly-star convergent to $x^*$ and $\|x^*\|\leq\ell$. By the maximal monotonicity of $\partial f$, we have $x^*\in\partial f(x)$ (see \cite[Fact 3.5]{BorweinYao13} or \cite[Section 2, page 539]{BFG03}) and so 
$\partial f(x)\cap\ell\mathbb{B}^*\ne\emptyset$. By Theorem~\ref{LipschitzBoundary},  $f$ is $\ell-$Lipschitz on $S$.
$\hfill\Box$
\end{proof}

\medskip
The following corollary is a direct application of Theorem\ref{LipschitzBoundary}. It gives the upper bound for the difference of the supremum and the infimum of a proper l.s.c. convex function on an open bounded set  in terms of the Lipschitz constant and diameter of the given set. 
\begin{Corollary} Let $f: X \to \overline{\mathbb{R}}$ be a proper l.s.c. convex function, $\ell\geq 0$ and $\overline{S}\subset\operatorname{dom}f$  be a nonempty bounded set. If there exists $\ell\geq 0$ such that
$\partial f(x)\cap\ell\mathbb{B}^*\ne\emptyset$ for all $x\in\operatorname{bd}(S)$ then 
\begin{equation}\label{supinf}
\sup_{x\in S}f(x)\leq\inf_{x\in S}f(x)+\ell \operatorname{diam}(S),
\end{equation}
where $\operatorname{diam}(S):=\sup\{\|s_1-s_2\|:s_1,s_2\in S\}$ is the diameter of $S$. 
\end{Corollary}
\begin{proof} According to Theorem~\ref{LipschitzBoundary}, $f$ is $\ell-$Lipschitz on $S$. It follows that, for every $x,y\in S$,
	$$
	f(x)\leq f(y)+\ell\|x-y\|\leq f(y)+\ell\operatorname{diam}(S). 
	$$
Taking the supremum in the variable $x$ of the left hand side and infimum in the variable $y$ 
of the right hand side of the above inequality, we get \eqref{supinf}.
$\hfill\Box$
\end{proof}

\medskip
We end this section by using Theorem~\ref{LipschitzBoundary} to deduce a criterion of the global Lipschitz continuity from the boundedness of the distance function from the origin to the subdifferential operator at infinity. 
\begin{Corollary} Let $f: X \to\mathbb{R}$ be a proper l.s.c. convex function and $\ell\geq 0$.
	Then, $f$ is $\ell-$Lipschitz continuous on $X$ if and only if 
	\begin{equation}\label{Minimal}
	\limsup_{\|x\|\rightarrow+\infty} d(0, \partial f(x)) \leq\ell,
	\end{equation}
	where $d(0, \partial f(x)):= \inf\{\Vert x^*\Vert:\, x^* \in \partial f(x)\}.$
\end{Corollary}
\begin{proof} If $f$ is $\ell-$Lipschitz continuous on $X$ then for every $x\in X$, $\partial f(x)\ne\emptyset$
	and $\partial f(x)\subset\ell\mathbb{B}^*$, and so  \eqref{Minimal} holds.
	Conversely, suppose that \eqref{Minimal} holds. Let $x,y$ be two any points in $X$. By \eqref{Minimal}, for every $\varepsilon>0$ there exists $r>0$ (sufficiently large) such that $x,y\in r\mathbb{B}$ and
	$\partial f(z) \cap (\ell+\varepsilon)\mathbb{B}^* \ne \emptyset$ 
	for all $z\in X$ such that $\|z\|=r$. 
	By Theorem~\ref{LipschitzBoundary}, $f$ is $(\ell+\varepsilon)-$Lipschitz on $r\mathbb{B}$ and so 
	$$
	|f(x)-f(y)|\leq(\ell+\varepsilon)\|x-y\|.
	$$
	Taking $\varepsilon\downarrow 0$ in the latter inequality, we get the conclusion.
	$\hfill\Box$
\end{proof}
\section{Applications}
Two classical problems in convex analysis are investigated in this section: the extension of a Lipschitz and convex function to the whole space and the justification of Lipschitz continuity of its Moreau envelope functions.
\subsection{Extension of Lipschitz and convex functions}
Let $\ell\geq 0$ and $f:X\rightarrow\overline{\mathbb{R}}$ be a proper convex and $\ell-$Lipschitz function
on $\operatorname{dom}f$. 
We say that $\tilde{f}:X\rightarrow\mathbb{R}$ is an extension of $f$ if $\tilde{f}$ is convex, $\ell-$Lipschitz continuous and satisfies
\begin{equation}\label{Relation}
\tilde{f}(x)=f(x)\;\;\text{for all}\;\; x\in\operatorname{dom}f.
\end{equation}
Observe that if $\tilde{f}$ is an extension of $f$ then
\begin{equation}\label{inclusion}
\partial\tilde{f}(x)\subset \partial f(x)\cap\ell\mathbb{B}^*\;\;\text{for all}\;\; x\in\operatorname{dom}f.
\end{equation}
Indeed, since $\tilde{f}$ is $\ell-$Lipschitz continuous on $X$, $\partial\tilde{f}(x)\subset\ell\mathbb{B}^*$
for all $x\in X$. Moreover, by \eqref{Relation}, we have $\partial\tilde{f}(x)\subset\partial f(x)$ for all $x\in\operatorname{dom}f$ . The inclusion \eqref{inclusion} may be strict. Indeed, consider the function $f(x)=\delta_{[0,1]}(x)+x$. Clearly, $f$ is convex and $1-$Lipschitz on $\operatorname{dom}f=[0,1]$. The function $\tilde{f}(x)=x$ is an extension of $f$ with $\partial \tilde{f}(0)=\{1\}$. However, $\partial f(0)\cap1\mathbb{B}^*=[-1,1]$.

\medskip
The next theorem constructs an extension of $f$ such that the inclusion \eqref{inclusion} becomes an equality. 
\begin{Theorem}\label{NewExtension} Let $\ell\geq 0$ and $f: X \to \overline{\mathbb{R}}$ be a proper convex and $\ell-$Lipschitz function on $\operatorname{dom}f$. Consider the function $\tilde{f}:X:\rightarrow\overline{\mathbb{R}}$ given by
\begin{equation}\label{EFunction}
	\tilde{f}(x):=\sup_{\substack{y\in \operatorname{dom}f\\ y^*\in \partial f(y)\cap\ell\mathbb{B}^*}}\left\{\langle y^*,x-y\rangle+f(y)\right\}, \quad \forall x\in X.
\end{equation}
	Then,  $\operatorname{dom}\tilde{f}=X$ and $\tilde{f}$ is a convex and $\ell-$Lipschitz function on $X$ and for every  $x\in\operatorname{dom}f$
\begin{equation}\label{equality}
	\tilde{f}(x)=f(x), \quad \partial\tilde{f}(x)=\partial f(x)\cap\ell\mathbb{B}^*.
\end{equation} 	
\end{Theorem}
\begin{proof} By Corollary~\ref{Global}, for every $x\in\operatorname{dom}f$, $\partial f(x)\cap\ell\mathbb{B}^*\ne\emptyset$ and so $f(x)\leq\tilde{f}(x)$. Let
	$$
	S:=\{(y,y^*)\in\operatorname{dom}f\times X^*:y^*\in\partial f(y)\cap\ell\mathbb{B}^*\}.
	$$
Clearly, $S\ne\emptyset$ and $\tilde{f}$ is given by
$$
\tilde{f}(x)=\sup_{(y,y^*)\in S}\left\{\langle y^*,x-y\rangle+f(y)\right\}, \quad \forall x\in X.
$$	
Since $\tilde{f}$ is the pointwise supremum of linear functionals, $\tilde{f}$ is convex. 
Let $x,u$ be two points in $X$. For every $(y,y^*)\in S$, by the definition of subdifferential
\begin{equation}\label{eq1}
	\langle y^*,x-y\rangle+f(y)\leq f(x),
\end{equation}
and by  the inequality $\|y^*\|\leq\ell$, we have
\begin{equation}\label{eq2}
\begin{split}
	\langle y^*,x-y\rangle+f(y)&=\langle y^*,x-u\rangle+\langle y^*,u-y\rangle+f(y)\\
	&\leq\|y^*\|\|x-u\|+\langle y^*,u-y\rangle+f(y)\\
	&\leq\ell\|x-u\|+\langle y^*,u-y\rangle+f(y).
\end{split}
\end{equation}	
Taking the supremum to all $(y,y^*)\in S$ both sides of \eqref{eq1} and \eqref{eq2}, we get
\begin{equation}\label{eq3}
\tilde{f}(x)\leq f(x), \quad \tilde{f}(x)\leq\ell\|x-u\|+\tilde{f}(u).
\end{equation}
The first inequality in \eqref{eq3} implies that $f(x)=\tilde{f}(x)$ for every $x\in\operatorname{dom}f$.
Hence, the second inequality in \eqref{eq3} implies that $\operatorname{dom}\tilde{f}=X$ and $\tilde{f}$ is $\ell-$Lipschitz on $X$. By \eqref{inclusion}, to get \eqref{equality}, it suffices to show that 
$\partial f(x)\cap\ell\mathbb{B}^*\subset \partial \tilde{f}(x)$ for all $x\in\operatorname{dom}f$. Indeed, let $x\in\operatorname{dom}f$ and $x^*\in \partial f(x)$. Then, $f(x)=\tilde{f}(x)$ and by \eqref{EFunction}, for every $z\in X$, we have 
\begin{eqnarray*}
	\tilde{f}(z)&\geq& f(x)+\langle x^*,z-x\rangle\\
	&=&  \tilde{f}(x)+\langle x^*,z-x\rangle.
\end{eqnarray*}
and so $x^*\in\partial\tilde{f}(x)$. Hence, $\partial f(x)\subset \partial \tilde{f}(x)$.
$\hfill\Box$
\end{proof}

\medskip
When the domain of the function has a nonempty interior and the original function is lower semicontinuous, the extended function can be refined. 
\begin{Corollary} Let $\ell\geq 0$ and $f: X \to \overline{\mathbb{R}}$ be a proper l.s.c. convex and $\ell-$Lipschitz function on $\operatorname{dom}f$. If $\operatorname{dom}f$ has a non-empty interior, then the function $\tilde{f}:X\rightarrow\mathbb{R}$ given by \eqref{EFunction} can be expressed as
	\begin{equation}
	\tilde{f}(x)=
	\begin{cases}
	f(x) & \text{ if } x \in \operatorname{dom}f,\\
	\underset{\substack{y\in \operatorname{bd}(\operatorname{dom}f)\\ y^*\in \partial f(y)\cap\ell\mathbb{B}^*}}{\sup}\left\{\langle y^*,x-y\rangle+f(y)\right\} & \text{ otherwise.}
	\end{cases}
	\end{equation}
\end{Corollary}
\begin{proof} Observe that $\operatorname{dom}f$ is closed. Indeed, suppose that $\{x_n\}\subset \operatorname{dom}f$ and $x_n\rightarrow x$.  By the Lipschitz continuity of $f$, for every $n\in\mathbb{N}$,
\begin{eqnarray*}
	f(x_n)&\leq& |f(x_n)-f(x_1)|+|f(x_1)|\\
	&\leq&\ell\|x_n-x_1\|+|f(x_1)|.
\end{eqnarray*}
By the lower semicontinuity of $f$, we have
$$
f(x)\leq\liminf_{n\rightarrow\infty}f(x_n)\leq \ell\|x-x_1\|+|f(x_1)|
$$
and so $x\in\operatorname{dom}f$. Hence, $\operatorname{dom}f$ is closed. 
	According to Theorem \ref{NewExtension}, it suffices show that for any $x \notin \operatorname{dom}f$ and $z \in \operatorname{dom}f, z^* \in \partial f(z)\cap \ell\mathbb{B}^*$, then
	\begin{equation}\label{e.1}
	\langle z^*,x-z\rangle+f(z) \le \underset{\substack{y\in \operatorname{bd}(\operatorname{dom}f)\\ y^*\in \partial f(y)\cap\ell\mathbb{B}^*}}{\sup}\left\{\langle y^*,x-y\rangle+f(y)\right\}.
	\end{equation}
	The inequality automatically holds if $z \in \operatorname{bd}(\operatorname{dom}f)$. Now we assume that $z \in \operatorname{int}(\operatorname{dom}f)$. Since $\operatorname{dom}f$ is closed,  there exists $\alpha \in (0,1)$ such that 
	$$
	y:= z+\alpha(x-z) \in \operatorname{bd}(\operatorname{dom}f)\subset\operatorname{dom}f.
	$$ By Corollary \ref{Global},  $\partial f(y)\cap \ell\mathbb{B}^* \ne \emptyset$. Pick $y^* \in \partial f(y) \cap \ell\mathbb{B}^*$. By the definition and the monotonicity of $\partial f$, we have
	\begin{equation*}
	\begin{split}
	f(y) +\langle y^*, x-y\rangle & \ge f(z) +\langle z^*, y-z \rangle + \langle y^*, x-y \rangle \\
	& \ge f(z) + \langle z^*, y-z \rangle +\frac{1-\alpha}{\alpha}\langle y^*,y-z \rangle \\
	& \ge f(z) + \langle z^*, y-z \rangle +\frac{1-\alpha}{\alpha}\langle z^*,y-z \rangle \\
	& = f(z) + \langle z^*, x-z \rangle,
	\end{split}
	\end{equation*}
	which implies that \eqref{e.1} is satisfied.
$\hfill\Box$
\end{proof}

\subsection{Moreau envelopes of the convex functions}
Let us recall the notion of Moreau envelope of a function and its properties. 
For any $\lambda>0$ the Moreau envenlope of index $\lambda$ of a function $f:X\rightarrow\overline{\mathbb{R}}$ is defined by
$$
	f_\lambda(x):=\inf\left\{f(y)+\frac{\|x-y\|^2}{2\lambda}:y\in X\right\}, \quad\forall x\in X.
$$
Suppose now that $f$ is a proper l.s.c. convex function.  In this setting, the envelopes of $f$ are also called Moreau-Yosida regularizations. 
By \cite[Proposition~1.10]{Brezis11}, $f$ is bounded below by an affine continuous functional, i.e., there exist $x^*\in X^*$ and $\alpha\in\mathbb{R}$ such that
\begin{equation}\label{LowerBounded}
f(y)\geq \langle x^*,y\rangle+\alpha, \quad \forall y\in X. 
\end{equation}
The envelopes of $f$ have some remarkable properties:
\begin{itemize}
	\item Since $f_\lambda$ is the infimum convolution of two convex functions $f$ and $y\mapsto(1/2\lambda)\|y\|^2$,  $f_\lambda$ is also convex by \cite[Theorem~2.1.3(ix)]{Zalinescu02}. 
	\item Clearly, by the properness of $f$ and \eqref{LowerBounded},  $f_\lambda$ has real values on $X$, or equivalently, $\operatorname{dom}f_\lambda=X$. Moreover, since $f_\lambda$ is bounded from above on a neighborhoood of every point in $\operatorname{dom}f$ and $\operatorname{int}(\operatorname{dom}f_\lambda)=X\ne\emptyset$, it is locally Lipschitz continuous on $X$ by \cite[Proposition~2.107]{BS00}.
	\item From \eqref{LowerBounded}, we can choose $x_0\in\operatorname{dom}f$ and $r>0$ such that 
	$$
	f(y)+r(\|y-x_0\|^2+1)\geq 0, \quad \forall y\in X.
	$$
	Therefore, by \cite[Theorem~2.64]{Attouch84}, for every $x\in X$, the sequence $\{f_\lambda(x)\}_{\lambda>0}$ increases to $f(x)$ as $\lambda$ decreases to zero and
	\begin{equation}\label{limit}
	f(x)=\lim_{\lambda\downarrow 0}f_\lambda(x)=\sup_{\lambda>0}f_\lambda(x).
	\end{equation}
	Moreover, if $f$ is $\ell-$Lipschitz on $\operatorname{dom}f$ then $\{f_\lambda\}$ uniformly converges to $f$ on $\operatorname{dom}f$. Indeed, for any $x,y\in\operatorname{dom}f$, we have
	\begin{eqnarray*}
		f(y)+\frac{1}{2\lambda}\|x-y\|^2&\geq& f(x)-\ell\|x-y\|+\frac{1}{2\lambda}\|x-y\|^2\\
		&=& f(x)+\frac{1}{2\lambda}\left(\|x-y\|-\lambda\ell\right)^2-\frac{\lambda\ell^2}{2}\\
		&\geq& f(x)-\frac{\lambda\ell^2}{2}. 
	\end{eqnarray*}
Taking the infimum in the variable $y$ both sides of the latter inequality, we get
$$
f(x)\geq f_\lambda(x)\geq f(x)-\frac{\lambda\ell^2}{2}.
$$
The latter inequalities show that $\{f_\lambda\}$ uniformly converges to $f$ on $\operatorname{dom}f$.
	\item Consider the function $g:X\rightarrow\mathbb{R}$ given by
	\begin{equation}\label{InFunction}
	g(y):=f(y)+\frac{\|x-y\|^2}{2\lambda}, \quad \forall y\in X.
	\end{equation}
	Since $f$ is a proper l.s.c. convex function, so is $g$. By \eqref{LowerBounded}, $g$  is coercive, i.e.,  
	$$
	\lim_{\|y\|\rightarrow+\infty}g(y)=+\infty.
	$$
	Hence, if $X$ is reflexive then $g$ attains its minimum on $X$ by \cite[Corollary~3.23]{Brezis11}.  This implies that, for every $x\in X$ and $\lambda>0$, there exists $x_\lambda\in\operatorname{dom}f$ such that
	\begin{equation}\label{Exact}
	f_\lambda(x)=\min_{y\in X}\left[f(y)+\frac{1}{2\lambda}\|x-y\|^2\right]=f(x_\lambda)+\frac{\|x-x_\lambda\|^2}{2\lambda}.
	\end{equation}
	Moreover, by \eqref{Exact}, applying the subdifferential rule for the infimum convolution \cite[Corollary~2.4.7]{Zalinescu02} of two convex functions $f$ and $y\mapsto(1/2\lambda)\|y\|^2$, we obtain
	\begin{equation}\label{SubMoreau}
	\partial f_\lambda(x)=\partial f(x_\lambda)\cap(1/\lambda)J(x-x_\lambda),
	\end{equation}
	where $J$ is the subdifferential of the convex function $x\mapsto(1/2)\|x\|^2$ given by
	\begin{equation}\label{DualityMapping}
	J(x)=\{x^*\in X^*:\langle x^*,x\rangle=\|x^*\|^2=\|x\|^2\}.
	\end{equation}
\end{itemize}
It follows from \eqref{limit} that if $f_\lambda$ is $\ell-$Lipschitz on $\operatorname{dom}f$  for all sufficiently small $\lambda>0$ then $f$ is also $\ell-$Lipschitz on its domain. Our aim is to prove the reverse result. 
We first establish the upper bound for the values of the Moreau envelope of $f$ at points of subdifferentiability. 
\begin{Lemma}\label{Bound}
	Let $f: X \to \overline{\mathbb{R}}$ be a proper l.s.c. convex function and $\lambda>0$. If $x \in \operatorname{dom}(\partial f)$ and $x^* \in \partial f(x)$ then 
	\begin{equation}\label{UpperBound}
 f_{\lambda}(x)\leq	f_{\lambda}(y)+\Vert x^*\Vert \left\Vert y-x\right\Vert,\quad  \forall y \in X.
	\end{equation}
\end{Lemma}
\begin{proof} Suppose that $x \in \operatorname{dom}(\partial f)$ and $x^* \in \partial f(x)$.   Let $\varepsilon>0$ be arbitrary and $x_\varepsilon\in\operatorname{dom}f$ be such that 
	$$
	f(x_\varepsilon)+\frac{1}{2\lambda}\|x_\varepsilon-x\|^2\leq f(y)+\frac{1}{2\lambda}\|y-x\|^2+\varepsilon^2, \quad \forall y\in X.
	$$
	By Br{\o}ndsted-Rockafellar theorem \cite{BR65}, there exist $y_\varepsilon\in X$ and $y^*_\varepsilon\in X^*$ such that $\|y^*_\varepsilon\|\leq\varepsilon$ and $y^*_\varepsilon\in \partial g(y_\varepsilon)$ where $g:X\rightarrow\mathbb{R}$ given by  \eqref{InFunction}. By Theorem~\ref{SumRule}, $y^*_\varepsilon\in \partial f(y_\varepsilon)+(1/\lambda)J(y_\varepsilon-x)$.  Let $z^*_\varepsilon\in \partial f(y_\varepsilon)$ be such that $y^*_\varepsilon-z^*_\varepsilon\in (1/\lambda)J(y_\varepsilon-x)$. By \eqref{DualityMapping}, the monotonicity of $\partial f$ and the inequality $\|y^*_\varepsilon\|\leq\varepsilon$, we have
	\begin{eqnarray*}
		\|y_\varepsilon-x\|^2&=&\lambda\langle y^*_\varepsilon-z^*_\varepsilon,y_\varepsilon-x \rangle\\
		&=&\lambda\langle y^*_\varepsilon-x^*,y_\varepsilon-x \rangle-\lambda\langle x^*-z^*_\varepsilon,x-y_\varepsilon\rangle\\
		&\leq&\lambda\langle y^*_\varepsilon-x^*,y_\varepsilon-x \rangle\\
	&\leq&\lambda\|y^*_\varepsilon-x^*\|\|y_\varepsilon-x\|\\
		&\leq&\lambda(\|y^*_\varepsilon\|+\|x^*\|)\|y_\varepsilon-x\|\\
		&\leq&\lambda(\varepsilon+\|x^*\|)\|y_\varepsilon-x\|.
	\end{eqnarray*}
It follows that
\begin{equation}\label{Bounded}
\|y_\varepsilon-x\|\leq \lambda\varepsilon+\lambda\|x^*\|. 
\end{equation}
Let $y,z$ be arbitrary  in $X$. Then, by \eqref{DualityMapping} and \eqref{Bounded}, we have
\begin{eqnarray*}
\langle z^*_\varepsilon-y^*_\varepsilon, z-x\rangle&\leq&\|z^*_\varepsilon-y^*_\varepsilon\|\|z-x\|\\
&=&(1/\lambda)\|y_\varepsilon-x\|\|z-x\|\\
&\leq&(1/\lambda)\|y_\varepsilon-x\|\|z-y\|+(1/\lambda)\|y_\varepsilon-x\|\|y-x\|\\
&\leq&(1/\lambda)\|y_\varepsilon-x\|\|z-y\|+(\varepsilon+\|x^*\|)\|y-x\|,
\end{eqnarray*}
\begin{eqnarray*}
\langle y^*_\varepsilon, z-y_\varepsilon\rangle&\leq&\|y^*_\varepsilon\|\|z-y_\varepsilon\|\\
&\leq&\varepsilon(\|z-x\|+\|x-y_\varepsilon\|)\\
&\leq&\varepsilon(\|z-x\|+\lambda\varepsilon+\lambda\|x^*\|).	
\end{eqnarray*}
Since $z^*_\varepsilon\in\partial f(y_\varepsilon)$, using \eqref{Bounded} again and the latter inequalities, we get
\begin{eqnarray*}
	f(z)&\geq&f(y_\varepsilon)+\langle z^*_\varepsilon, z-y_\varepsilon\rangle\\
	&=&f(y_\varepsilon) +\langle z^*_\varepsilon-y^*_\varepsilon, x-y_\varepsilon\rangle+\langle z^*_\varepsilon-y^*_\varepsilon, z-x\rangle+\langle y^*_\varepsilon, z-y_\varepsilon\rangle\\
	&\geq&	f(y_\varepsilon) +\frac{1}{\lambda}\|x-y_\varepsilon\|^2-\frac{1}{\lambda}\|y_\varepsilon-x\|\|z-y\|-(\varepsilon+\|x^*\|)\|y-x\|-\varepsilon(\|z-x\|+\lambda\varepsilon+\lambda\|x^*\|)\\
	&=&\left[f(y_\varepsilon)+\frac{1}{2\lambda}\|x-y_\varepsilon\|^2\right]+\left[\frac{1}{2\lambda}\|x-y_\varepsilon\|^2-\frac{1}{\lambda}\|y_\varepsilon-x\|\|z-y\|\right]-\|x^*\|\|y-x\|\\
	&&-\varepsilon(\|y-x\|+\|z-x\|+\lambda\varepsilon+\lambda\|x^*\|).
\end{eqnarray*}
On the other hand, by the definition of Moreau envelope and the Cauchy inequality, we have  	
$$
f(y_\varepsilon)+\frac{1}{2\lambda}\|x-y_\varepsilon\|^2\geq f_\lambda(x),
$$
$$
\frac{1}{2\lambda}\|x-y_\varepsilon\|^2-\frac{1}{\lambda}\|y_\varepsilon-x\|\|z-y\|\geq -\frac{1}{2\lambda}\|y-z\|^2.
$$
Therefore,
$$
f(z)+\frac{1}{2\lambda}\|y-z\|^2\geq f_\lambda(x)-\|x^*\|\|y-x\|-\varepsilon(\|y-x\|+\|z-x\|+\lambda\varepsilon+\lambda\|x^*\|).
$$
Letting $\varepsilon\downarrow 0$ in the above inequality, we get
$$
f(z)+\frac{1}{2\lambda}\|y-z\|^2\geq f_\lambda(x)-\|x^*\|\|y-x\|.
$$
It follows that
$$
f_\lambda(y)=\inf_{z\in X}\left[f(z)+\frac{1}{2\lambda}\|y-z\|^2\right]\geq f_\lambda(x)-\|x^*\|\|y-x\|
$$
and so \eqref{UpperBound} is satisfied.
	$\hfill\Box$
\end{proof}
\begin{Theorem} \label{LipschitzMoreau}
	Let $f: X \to \overline{\mathbb{R}}$ be a proper l.s.c. convex function and $\lambda>0$. 
	If $f$ is $\ell-$Lipschitz on $\operatorname{dom}f$ then $f_\lambda$ is $\ell-$Lipschitz on $\operatorname{dom}f$.
\end{Theorem}
\begin{proof} Let $x,y$ be arbitrary  in $\operatorname{dom}f$. Since $f$ is $\ell-$Lipschitz on $\operatorname{dom}f$, by Corollary~\ref{Global}, $\operatorname{dom}f=\operatorname{dom}\partial f$ and there exist $x^*\in \partial f(x), y^*\in\partial f(y)$ such that $\|x^*\|\leq\ell$ and $\|y^*\|\leq\ell$. Applying Lemma~\ref{Bound}, we get
	$$
	f_{\lambda}(x)\leq	f_{\lambda}(y)+\Vert x^*\Vert \left\Vert y-x\right\Vert\leq f_{\lambda}(y)+\ell\|y-x\|,
	$$
	$$
	f_{\lambda}(y)\leq	f_{\lambda}(x)+\Vert y^*\Vert \left\Vert x-y\right\Vert\leq 	f_{\lambda}(x)+\ell\|x-y\|.
	$$
	Therefore, 
	$$
	\|f_{\lambda}(x)-f_{\lambda}(y)\|\leq\ell\|x-y\|,
	$$
	and so $f_\lambda$ is $\ell-$Lipschitz on $\operatorname{dom}f$.
	$\hfill\Box$
\end{proof}
\begin{Remark} {\rm When $X$ is reflexive we can give a simple proof of Theorem~\ref{LipschitzMoreau} without using Lemma~\ref{Bound}. Indeed, let $\lambda>0$ and $x\in \operatorname{dom}f$. Since $X$ is reflexive, there exists $x_\lambda\in\operatorname{dom}f$ satisfying  \eqref{SubMoreau}.	
Moreover, since $f_\lambda$ is convex and locally Lipschitz on $X$, we have  $\partial f_\lambda(x)\ne\emptyset$.  Let $x^*_\lambda\in\partial f_\lambda(x)$. By \eqref{SubMoreau}, we have
$x^*_\lambda\in\partial f(x_\lambda)\cap(1/\lambda)J(x-x_\lambda)$. 
Since $f$ is $\ell-$Lipschitz on $\operatorname{dom}f$, by Corollary~\ref{normal}, $x^*_\lambda\in\partial f(x_\lambda)\subset N_\ell(x_\lambda;\operatorname{dom}f)$ and so 
\begin{equation}\label{I1}
\langle x^*_\lambda,x-x_\lambda\rangle\leq \ell\|x-x_\lambda\|. 
\end{equation}
On the other hand, $x^*_\lambda\in (1/\lambda)J(x-x_\lambda)$ implies that
\begin{equation}\label{I2}
\langle x^*_\lambda,x-x_\lambda\rangle=(1/\lambda)\|x-x_\lambda\|^2=\lambda\|x^*_\lambda\|^2.
\end{equation}
Combining \eqref{I1} and \eqref{I2}, we get
$$
\|x^*_\lambda\|=(1/\lambda)\|x-x_\lambda\|\leq \ell.
$$
 Hence, $\partial f_\lambda (x)\ne\emptyset$ and $\partial f_\lambda (x)\subset \ell\mathbb{B}^*$ for all $\lambda>0$ and 
$x\in \operatorname{dom}f$. By Lemma~\ref{CharacterizationSet}, $f_\lambda$ is $\ell-$Lipschitz on $\operatorname{dom}f$ for all $\lambda>0$.
		}
\end{Remark}
\section{Conclusions}
Characterizations of Lipschitz continuity of a proper l.s.c. convex function on a real Banach space are investigated in this paper. Our criteria are expressed in terms of the boundedness of some selection of the subdifferential operator and the intersections of the subdifferential operator  and the normal cone operator to  domain of the function in question. On an open and bounded (not necessarily convex) set, we also characterize the Lipschitz continuity of the given function via the boundedness a selection of the subdifferential operator on the boundary of the given set. Applications to two classical problems in convex analysis are given:  the extension of a Lipschitz and convex function to the whole space and the justification of Lipschitz continuity of its Moreau envelope functions.


\begin{thebibliography}{99}
\bibitem{Attouch84} Attouch, H.: Variational Convergence for Functions and Operators. Applicable Mathematics Series, Pitman, London (1984)

\bibitem{Ba}  Barbu, V.: Nonlinear Differential Equations of Monotone Types in Banach Spaces.Springer Monographs in Mathematics. Springer, New York (2010)

\bibitem{BS00} Bonnans, J.F., Shapiro, A.: Perturbation Analysis of Optimization Problems. Springer, New York (2000)

\bibitem{BFG03} Borwein, J.M.,  Fitzpatrick, S., and  Girgensohn, R.: \emph{Subdifferentials whose graphs are not norm$\times$weak$^*$ closed}.  Canad. Math. Bull. \textbf{46}, 538--545 (2003)

\bibitem{BorweinYao13} Borwein, J.M.,  Yao, L.: \emph{Structure theory for maximally monotone operators with points of continuity}. J. Optim. Theory Appl. \textbf{157}, 1--24 (2013)

	\bibitem{Brezis11} Br\'ezis H.: Functional Analysis, Sobolev Spaces and Partial Differential Equations. Universitext. Springer, New York (2011)
	
	\bibitem{BR65} Br{\o}ndsted, A., Rockafellar, R.T.: \emph{On the subdifferentiability of convex functions}. Proc. Amer. Math. Soc. \textbf{16}, 605--611 (1965)

\bibitem{Hiriat80} Hiriart-Urruty, J.-B.:  \emph{Extension of Lipschitz functions}. J. Math. Anal. Appl. \textbf{77}, 539--554 (1980) 


\bibitem{Mordukhovich06} Mordukhovich, B.S.: Variational Analysis and Generalized Differentiation. In: Basic Theory, vol. I. Applications, vol. II. Springer, Berlin (2006)

\bibitem{MMS19}
Mohammadi, A., Mordukhovich, B.S., and Sarabi, M.E.: \emph{Parabolic Regularity in Geometric Variational Analysis}. https://arxiv.org/pdf/1909.00241.pdf

\bibitem{Penot13}
Penot, J.-P.: Calculus Without Derivatives. Graduate Texts in Mathematics. Springer, New York (2013)

\bibitem{Rockafellar66} Rockafellar,  R.T.: \emph{Extension of Fenchel's duality theorem for convex functions.} Duke Math. J. \textbf{33},  81--89 (1966) 

\bibitem{Rockafellar70} Rockafellar, R.T. : \emph{On the maximal monotonicity of subdifferential mappings}. Pacific J. Math. \textbf{33}, 209--216 (1970) 

\bibitem{RockafellarWets98} Rockafellar, R.T., Wets, R.J.-B.: Variational Analysis. Springer, Berlin (1998)


\bibitem{Zagrodny88} Zagrodny,  D.: \emph{Approximate mean value theorem for upper subderivatives}. Nonlinear Anal. \textbf{12}, 1413--1428 (1988) 

\bibitem{Zalinescu02} Z\u{a}linescu, C.: Convex Analysis in General Vector Spaces. World Scientific, Singapore (2002)
\end{thebibliography}
\end{document}